\documentclass[a4paper,11pt,leqno]{article}
\usepackage{amsmath,amsfonts,amssymb,amsthm}
\usepackage{graphicx,subfigure,booktabs,multirow}
\usepackage{epstopdf}
\usepackage{color}
\usepackage{hyperref}
\usepackage{aliascnt}
\usepackage[textwidth=15.2cm]{geometry}
\newtheorem{theorem}{Theorem}[section]
\newaliascnt{lemma}{theorem}
\newtheorem{lemma}[lemma]{Lemma}
\aliascntresetthe{lemma}
\newtheorem{proposition}[theorem]{Proposition}
\newaliascnt{corollary}{theorem}
\newtheorem{corollary}[corollary]{Corollary}
\aliascntresetthe{corollary}

\newtheorem{definition}[theorem]{Definition}

\newcommand{\be}{\begin{equation}}
\newcommand{\ee}{\end{equation}}
\newcommand{\bes}{\begin{equation*}}
\newcommand{\ees}{\end{equation*}}
\newcommand{\ba}{\begin{eqnarray}}
\newcommand{\ea}{\end{eqnarray}}
\newcommand{\bas}{\begin{eqnarray*}}
\newcommand{\eas}{\end{eqnarray*}}
\numberwithin{equation}{section}
\def\int{{\rm int}}

\title{The Alon--Tarsi Number of A Toroidal Grid }

\author{ Zhiguo Li$^1$\thanks{Corresponding author. E-mail: zhiguolee@hebut.edu.cn},~
 Zeling Shao$^1$, Fedor Petrov$^2$, Alexey Gordeev$^3$\\
{\small 1. School of Science, Hebei University of Technology, Tianjin 300401, China}
\footnote{This work is supported by  the Natural Science Foundation of Hebei Province,China (No. A2019402043). }\\
{\small ~2. Saint Petersburg Department of the Steklov Mathematical Institute, St. Petersburg, Russia}\\
{\small ~3. The Euler International Mathematical Institute, St. Petersburg, Russia}
\date{}
}

\begin{document}

\maketitle

%%%%%%%%%%%%%%%%%%%%%%%%%%%%%%%%%%%%%%%%%%%%%%%%%%%%%%%%%%%%%%%%%%%%%%%%%%%%%%%%%%%%%%%%%%%%%%%%%%%%%%%%%%%%%%%%%%
\begin{abstract}

The \emph{Alon--Tarsi number} $AT(G)$ of a graph $G$  is the smallest
$k$ for which there is an orientation  $D$  of $G$ with max indegree $k-1$ such that the
number of even and odd circulations contained in D are different.
  In this paper, we show that the Alon--Tarsi number of toroidal grids $T_{m,n}=C_m\Box C_n$ equals $4$  when $m,n$ are both odd and $3$ otherwise.


\vskip 12pt
  \noindent {\bf 2010 MR subject classifications. }{ 05C15, 05C31 } 
% 	65K15   	Numerical methods for variational inequalities and related problems

\end{abstract}
%\newpage

%%%%%%%%%%%%%%%%%%%%%%%%%%%%%%
\section{Introduction}

In this paper all graphs are finite, and all graphs are either simple graphs or simple
directed graphs. List coloring is a widely studied generalization of the classical notion of graph coloring.
List colorings of graphs, introduced independently by Vizing \cite{vizing1976} and  by
Erdos, Rubin and Taylor\cite{Erdos1979}, is usually known as the study of the choosability properties of a graph. Two excellent surveys on them are those by Alon \cite{alon1993} and Tuza \cite{tuza1997}, and the second of these has been updated by Kratochvil, Tuza and Voigt \cite{tuza97}.

Consider an arbitrary field $\mathbb{F}$ and a function $L$, which assigns a finite subset $L(v)\subseteq\mathbb{F}$ to each vertex $v$ of a graph $G$. An $L$-coloring of $G$ is a vertex coloring
$\phi$ that colors each vertex $v$ by a color $\phi(v)\in L(v)$ so that no adjacent vertices receive
the same color. A graph $G$ is $L$-colorable if it admits an $L$-coloring, and $k$-choosable
(or $k$-list-colorable) if it is $L$-colorable for every assignment of $k$-element lists to the
vertices. An $L$-coloring $\phi$ for $G$ with every vertex $v$ satisfying $|L(v)|=k$ is also called
a $k$-list-coloring. The \emph{list chromatic number }(or the choice number) $\chi_l(G)$ of a graph
$G$ is the smallest $k$ for which G is $k$-choosable.

In a seminal paper \cite{alon1993}, Alon and Tarsi have
introduced an algebraic technique for proving upper bounds on the
list chromatic number of graphs (and thus, in particular, upper bounds on
their chromatic number). The upper bound on the list chromatic number
of $G$ obtained via their method was later called the Alon--Tarsi
number of G and was denoted by $AT(G)$ (see e.g. Jensen and Toft
(1995) \cite{jensen2011}). They have provided a combinatorial interpretation
of this parameter in terms of the Eulerian subdigraphs of an
appropriate orientation of $G$. Their characterization can be restated
as follows.  The \emph{Alon--Tarsi number} of $G$, $AT(G)$, is the smallest
$k$ for which there is an orientation  $D$  of $G$ with max indegree $k-1$ such that the
number of even and odd circulations contained in D are different. There is an equivalent definition of Alon--Tarsi number by Alon--Tarsi polynomial method (see section 2). It follows from the Alon--Tarsi Theorem \cite{alon1993} that  $\chi(G)\leq \chi_l(G)\leq AT(G)$.

Let $G$ and $H$ be graphs. The \emph{Cartesian product} $G\Box H$ of $G$ and $H$ is the graph
with vertex set $V(G)\times V (H)$ where two vertices $(u_1, v_1)$ and $(u_2, v_2)$ are adjacent
if and only if either $u_1 = u_2$ and $v_1v_2\in E(H)$, or $v_1 = v_2$ and $u_1u_2 \in E(G)$.
Let $P_n$ and $C_n$ denote respectively the path and the cycle on $n$ vertices. We
will denote by $G_{m,n} = P_m\Box P_n$ the \emph{grid} with $m$ rows and $n$ columns and by
$T_{m,n} = C_m\Box C_n$ the \emph{toroidal grid} with $m$ rows and $n$ columns.
 In \cite{kaul2018}, H. Kaul and J. A. Mudrock showed that the Alon--Tarsi number of the Cartesian product of an odd cycle and a path is always equal to 3.
 L. Cai, W. Wang, and X. Zhu considered toroidal   grids $T_{m,n}$, where $m, n\geq 3$ and
 conjectured that  every toroidal grid is 3-choosable \cite{zhu10toroidal}.  It is easy to prove by
induction that $T_{m,n}$ is 3-colorable for all $m, n\geq 3$.
In this paper, we determine the Alon--Tarsi number  of  a toroidal grid $T_{m,n}$ which generalizes the result of  \cite{zhiguoliAT}:

$$AT(T_{m,n})=\left\{\begin{array}{ll}
   4, &n, m\text{~both odd},\\
   3, &\text{else}.
   \end{array}\right.$$

  As a byproduct, when $m,n$ are both even or one of them is odd and another is even, we get $\chi_l(T_{m,n})=3$. These support the conjecture in a positive way. When $m, n $ are both odd,  the list coloring number  of toroidal grids $T_{m,n}$ is undetermined although it is easy to get that $3\leq\chi_l(T_{m,n})\leq 4$ in this case.

The paper is organized as follows. In Section 2 we give basic properties and
illustrate the techniques we shall use in the proof of our main result, which is
given in Section 3. Finally we  discuss  some open problems in Section 4.

\section{Preliminaries}

Let $G = (V , E)$ be an undirected simple graph with vertex set $\{1,\ldots,n\}$. The graph polynomial of $G$
is defined by
$$f_G(x_1,x_2,\ldots,x_n)=\prod_{1\leq i<j\leq n, (i,j)\in E}(x_i-x_j).$$

It is clear that the graph polynomial encodes information about its proper colorings. Indeed,
a graph $G$ is $k$-colorable if and only if there exists an $n$-tuple $(a_1,a_2,\ldots,a_n)\in \{0, 1,\ldots,k-1\}^n$
such that $f_G(a_1,a_2,\ldots,a_n)\neq 0$. Similarly, $G$ is $k$-choosable if and only if for an arbitrary field $\mathbb{F}$ and for every family of sets
${S_i \subset \mathbb{F}: 1\leq i \leq n}$, each of size at least $k$, there exists an $n$-tuple $(a_1,a_2,\ldots,a_n)\in S_1\times S_2\times \cdots \times S_n$
such that $f_G(a_1,a_2,\ldots,a_n)\neq 0$.

The following theorem gives a sufficient condition for the existence of such an $n$-tuple.
\begin{theorem}[Combinatorial Nullstellensatz, \cite{Alon1999comb}]\label{thm:CN}
Let $\mathbb{F}$ be an arbitrary field, and let $f = f (x_1, \ldots , x_n)$
be a polynomial in $\mathbb{F}[x_1,\ldots , x_n]$. Suppose that the degree $\deg( f )$ of $f$ is
$\sum_{i=1}^{n}t_i$, where each $t_i$ is a nonnegative
integer, and suppose that the coefficient of $\prod_{i=1}^{n} x_i^{t_i}$
  in $f$ is non-zero. Then, if $S_1, \ldots, S_n$ are subsets
of $\mathbb{F}$ with $|S_i | >t_i$, then there are $s_1\in S_1, s_2 \in S_2, \ldots, s_n \in S_n $  so that
$f (s_1, \ldots, s_n)\neq 0$.
\end{theorem}

We are now ready to define the main concept of this paper.
\begin{definition}[\cite{hefetz11}]
Let $G = (V , E)$ be a graph with vertex set $V = \{v_1,\ldots, v_n\}$.
We say that $G$ is Alon--Tarsi $k$-choosable,  if there exists a monomial $c\prod_{i=1}^{n}x_i^{t_i}$
 in the expansion of $f_G$ such that $c \neq 0$ (say, in   $\mathbb{R}$) and $t_i\leq k- 1$ for every $1 \leq i \leq n$.   The smallest integer
$k$ for which $G$ is Alon--Tarsi $k$-choosable, denoted by $AT(G)$, is called the Alon--Tarsi number of $G$.
\end{definition}

According to  this definition  and Alon--Tarsi Theorem, it follows  that $AT(G)\geq \chi_l(G)$ for every graph $G$. The converse inequality, however, does not hold in general. Indeed, it was proved  in \cite{Erdos1979}  that $\chi_l(K_{n,n})=(1+o(1))\log_2n$. On the other hand, it is clear from the definiton that $AT(K_{n,n})\geq n/2$ (It was proved in \cite{Alon1992} that $AT(K_{n,n})= \lceil n/2 \rceil +1$ ).

Theoretically, \autoref{thm:CN} can be applied to many graph coloring and combinatorial problems. However, proving that some appropriate monomial does not vanish is often extremely hard. While, in the case of toroidal grids, we can get
the some information about the coefficients of some appropriate monomials by following coefficient formula.

Let $\mathbb{F}$ be an arbitrary field and let $A_1,\ldots,A_n$ be any finite subsets of $\mathbb{F}$. Define the function $N: A_1\times\cdots \times A_n\to\mathbb{F}$ by
$$ N(a_1,\ldots,a_n)=\prod_{i=1}^{n}\prod_{b\in A_i\setminus \{a_i\}} (a_i-b).$$

\begin{theorem}[Coefficient Formula, see \cite{petrov2012},\cite{Lason10comb}]
Suppose  a polynomial  $f(x_1,x_2,\ldots,x_n)$ over field $\mathbb{F}$ has degree at most $\sum_{i=1}^n t_i$ and let $[\prod_{i=1}^{n}x_i^{t_i}]f(x_1,x_2,\ldots,x_n)$ denote the coefficient of $x_1^{t_1}\cdots x_n^{t_n}$ in $f$. Then for any
sets $A_1,\ldots,A_n$ in $\mathbb{F}$ such that $|A_i|=t_i+1$ we have

$$\left[\prod_{i=1}^{n}x_i^{t_i}\right]f(x_1,x_2,\ldots,x_n)=\sum_{(a_1,\ldots,a_n)\in A_1\times\cdots \times A_n} \frac{f(a_1,\ldots,a_n)}{N(a_1,\ldots,a_n)}.$$
\end{theorem}

\section{Main results}

In \cite{zhiguoliAT}, the authors have showed the following:

\begin{proposition}[{\cite{zhiguoliAT}}]
$AT(T_{2m,2n} )=\chi_l(T_{2m,2n} )=3$, for $m\geq 2,n\geq 2.$.
\end{proposition}

\begin{lemma}\label{lem:trace}
Let $H=(V,E)$ be a $2d$-regular graph on the vertex set $X=\{v_1,\ldots,v_n\}$, $G=H\square C_k$. Define $N=nk$. Fix a field $\mathbb{F}$ and a subset $A\subset \mathbb{F}$, $|A|=d+2$. Let $\mathcal{U}$ denote the set of all proper $(d+2)$-colorings $u=(u_1,\ldots,u_n)\in A^n$ of the vertices of $H$ with colors taken from $A$. Consider the square matrix $M$ with rows and columns indexed by the elements of $\mathcal{U}$:
\[
M_{u,v}=f_H(u_1,\ldots,u_n)\cdot \prod_{i=1}^n \frac{u_i-v_i}{\prod_{b\in A\setminus\{u_i\}} (u_i-b)}
\]
for two proper colorings $u,v\in \mathcal{U}$. Then
$$
\left[\prod_{i=1}^N x_i^{d+1}\right] f_G(x_1,\dots,x_N)=\mathrm{tr}\,M^k,
$$
where the variables $x_i$ correspond to all $N$ vertices of $G$.
\end{lemma}

\begin{proof}
By the coefficient formula, we have the following:
\begin{equation}\label{eq:coeff}
\left[\prod_{i=1}^N x_i^{d+1}\right] f_G=\sum_{(a_1,\dots,a_N)\in A^N} \frac{f_G(a_1,\ldots,a_N)}{\prod_{i=1}^N \prod_{b\in A\setminus\{a_i\}} (a_i-b)}.
\end{equation}
The non-zero summands in the RHS of \eqref{eq:coeff} correspond to proper colorings of $G$. Any proper coloring of $G$ corresponds to a sequence $u^1,\ldots,u^k\in \mathcal{U}$ of the proper colorings of $H$. For such a sequence, the corresponding summand in RHS of \eqref{eq:coeff} reads as
$$
M_{u^1,u^2}\cdot M_{u^2,u^3}\cdot \ldots \cdot M_{u^k,u^1}.
$$
The sum of such products is exactly $\mathrm{tr}\, M^k$.
\end{proof}

\begin{corollary}\label{cor:oddeven}
Let $G$ denote $T_{2m+1,2n}$ for $m\geq 1$, $n\geq 2$, let $N=(2m+1)(2n)$. Then
\[
\left[\prod_{i=1}^N x_i^2 \right] f_G(x_1,\dots,x_N)\neq 0.
\]
\end{corollary}

\begin{proof}
We apply \autoref{lem:trace} with $H=C_{2m+1}$, $k=2n$. We choose the field $\mathbb{F}=\mathbb{C}$ and the set $A=\{1,w,w^2\}$, where $w=e^{2\pi i/3}$.

We need to prove that $\mathrm{tr}\, M^{2n}\ne 0$. To do that it is sufficient to show that $M$ is a non-zero antihermitian matrix. That implies that the eigenvalues of $M$ are purely imaginary and not all of them are equal to 0, so $(-1)^n\mathrm{tr}\, M^{2n}> 0$.

$M$ is certainly non-zero: for example, $M_{u,v}\neq 0$, where $u=(u_1,\dots,u_{2m+1})\in\mathcal{U}$ is any proper $3$-coloring of $C_{2m+1}$, and $v=(u_{2m+1},u_1,\dots,u_{2m})$.

It remains to prove that $M$ is antihermitian. Throughout the remaining proof we treat indices as cyclic modulo $2m+1$. Let $u=(u_1,\ldots,u_{2m+1})\in \mathcal{U}$ be a proper $3$-coloring of $C_{2m+1}$ with colors $1,w,w^2$. Denote by $u_i^*$ the unique element of the set $\{1,w,w^2\}\setminus \{u_i,u_{i-1}\}$. We apply the relation
\[
\frac1{u_i-u_{i}^*}=\frac{u_i-u_{i-1}}
{\prod_{b\in A\setminus\{u_i\}} (u_i-b)}
\]
to get
\begin{equation}\label{eq:uv*}
M_{u,v}=\prod_{i=1}^{2m+1}\frac{u_i-v_i}{u_i-u_i^*}.
\end{equation}
We need to check that for any $u,v\in\mathcal{U}$
\begin{equation}\label{eq:antiherm}
M_{u,v}=-\overline{M_{v,u}}.
\end{equation}
If $u_i=v_i$ for some $i$, then $M_{u,v}=0=-\overline{M_{v,u}}$. Otherwise applying \eqref{eq:uv*} and substituting $\bar{z}=1/z$ for roots of unity $z\in\{u_i,v_i,u_i^*\}$  we simplify \eqref{eq:antiherm} to the following:
\begin{equation}\label{eq:antiherm2}
\prod_{i=1}^{2m+1} \frac{u_i-u_i^*}{u_i}=-\prod_{i=1}^{2m+1} \frac{v_i^*-v_i}{v_i^*}.
\end{equation}
Denote $\varepsilon_i=u_i/u_{i-1}$, then $\varepsilon_i\in \{w,w^2\}$ and $u_i^*=u_i \varepsilon_i$. Similarly, denote $\delta_i=v_i/v_{i-1}$, so $\delta_i\in\{w,w^2\}$ and $v_i^*=v_i \delta_i$. By definition,
\begin{equation}\label{eq:epsprod}
\prod_{i=1}^{2m+1} \varepsilon_i=\prod_{i=1}^{2m+1} \delta_i=1.
\end{equation}
\eqref{eq:antiherm2} is equivalent to
\[
\prod_{i=1}^{2m+1}(1-\varepsilon_i)=-\prod_{i=1}^{2m+1} (1-\overline{\delta_i})=(-1)^{1+(2m+1)}\prod_{i=1}^{2m+1} \delta_i^{-1}\prod_{i=1}^{2m+1}(1-\delta_i).
\]
Using \eqref{eq:epsprod}, we rewrite this as
\begin{equation}\label{eq:antiherm3}
\prod_{i=1}^{2m+1}(1-\varepsilon_i)=\prod_{i=1}^{2m+1}(1-\delta_i).
\end{equation}
Note that $1-\varepsilon=\pm i\sqrt{3}\varepsilon^2$ for $\varepsilon\in \{w,w^2\}$; the signs are distinct for $w$ and $w^2$. Substituting this for $\varepsilon_i$'s and $\delta_i$'s and using \eqref{eq:epsprod}, we reduce \eqref{eq:antiherm3} to the following fact: the total number of $\varepsilon_i$'s and $\delta_i$'s which are equal to $w$ is even.

Call the index $i$ white if $u_i=w\cdot v_i$ and black if $u_i=w^2\cdot v_i$. Then $\varepsilon_i=\delta_i$ if $i-1$, $i$ have the same color and $\varepsilon_i\ne \delta_i$ if $i-1$, $i$ have different colors. To conclude the proof, note that there are even number of indices $i$ of the second type.
\end{proof}

\begin{theorem}\label{thm:oddeven}
$AT(T_{2m+1,2n} )=3$ for $m\geq 1,n\geq 2.$
\end{theorem}

% In \cite{kaul2018}, H. Kaul and J.A. Mudrock obtained  the list chromatic number of the
% Cartesian product of an arbitrary cycle and a path.

% \begin{theorem}[\cite{kaul2018}]\label{thm:cyclepath}
% $\chi_l(C_m\Box P_n)=3$, for $m\geq 3,n\geq 2.$
% \end{theorem}

\begin{corollary}
$\chi_l(T_{2m+1,2n} )=3$ for $m\geq 1,n\geq 2.$.
\end{corollary}
\begin{proof}
Since $C_{2m+1}$ is a subgraph of $T_{2m+1,2n}$, $\chi_l(T_{2m+1,2n})\geq\chi(C_{2m+1})\geq 3.$  Because  $\chi_l(T_{2m+1,2n})\leq AT(T_{2m+1,2n})$, the result is followed by \autoref{thm:oddeven}.
\end{proof}

\begin{corollary}\label{cor:oddodd}
Let $G$ denote $T_{2m+1,2n+1}$ for $m,n\geq 1$, let $N=(2m+1)(2n+1)$, then $$\left[\prod_{i=1}^{N}x_i^2\right]f_G(x_1,x_2,\ldots,x_{N})=0.$$
\end{corollary}
\begin{proof}
 In our situation $H=C_{2m+1}$, $k=2n+1$, by \autoref{lem:trace} we have $$\left[\prod_{i=1}^{N}x_i^2\right]f_G(x_1,x_2,\ldots,x_{N})=\mathrm{tr}\, M^{2n+1}.$$
Proceeding as in the proof of \autoref{cor:oddeven}, we have  that $M$ is a non-zero antihermitian matrix. That implies that the eigenvalues of $M$ are purely imaginary and not all of them are equal to 0, that yields $\mathrm{tr}\, M^{2n+1}$ is   purely imaginary. While  the coefficient $[\prod_{i=1}^{N}x_i^2]f_G(x_1,x_2,\ldots,x_{N})$ is actually a real number. Hence the coefficient vanishes.
\end{proof}

Recall a recent result:

\begin{theorem}[\cite{kaul2018}]\label{thm:kaul}
Suppose that $G$ is a complete graph or an odd cycle with $|V (G)|\geq 3$. Suppose
$H$ is a graph on at least two vertices that contains a Hamilton path, $w_1,w_2,\ldots ,w_m$, such
that $w_i$ has at most $k$ neighbors among $w_1,\ldots,w_{i-1}$. Then, $AT(G\Box H)\leq \Delta(G)+ k$.
\end{theorem}

Since $T_{2m+1,2n+1}=C_{2m+1}\Box C_{2n+1},$  according to \autoref{cor:oddodd} and \autoref{thm:kaul}, we get

\begin{theorem}
$AT(T_{2m+1,2n+1} )=4$, for $m,n\geq 1.$
 \end{theorem}

\section{Conclusions}
 In this paper, we have showed the following results:
$$AT(C_m\Box C_n)=\left\{\begin{array}{ll}
   4, &n, m\text{~both odd},\\
   3, &\text{else}.
   \end{array}\right.$$
$$\chi_l(C_m\Box C_n)=\left\{\begin{array}{ll}
   3~or ~4, &n, m\text{~both odd},\\
   3, &\text{else}.
   \end{array}\right.$$
  When  $n, m$ are both odd, $\chi_l(C_m\Box C_n)$ is  undetermined yet. Although its Alon--Tarsi number is $4$, we conjecture that      $\chi_l(C_{n}\Box C_{m})=3$.

%\section*{Acknowledgment}
%
%The authors wish to express their sincere appreciations to all those who made suggestions for improvements to this paper.
%WITHOUT CONCRETE NAMES SOUNDS STRANGE
%-FP.

%%%%%%%%%%%%%%%%%%%%%%%%%%%
 \bibliographystyle{plain}
\bibliography{graph}

\end{document}